\documentclass[12pt]{article}
\usepackage{enumerate}
\usepackage{amsmath}
\usepackage{amsthm}
\usepackage{amssymb}
\usepackage{algorithm}
\usepackage[bookmarksnumbered,  plainpages]{hyperref}

\newtheorem{thm}{Theorem}[section]

\newtheorem{coro}{Corollary}[section]
\newtheorem{lem}{Lemma}[section]
\newtheorem{mthd}{Method}[section]

\theoremstyle{definition}
\newtheorem{defn}{Definition}[section]
\theoremstyle{remark}

\begin{document}
	
	\title{{New Relaxation  Modulus Based Iterative Method for  Large and Sparse Implicit Complementarity  Problem}}
	
	\author{Bharat Kumar$^{a,1}$, Deepmala$^{a,2}$ and A.K. Das $^{b,3}$\\
		\emph{\small $^{a}$Mathematics Discipline, PDPM-Indian Institute of Information Technology,}\\
		\emph{\small  Design and Manufacturing, Jabalpur - 482005 (MP), India}\\
		\emph{\small $^{b}$Indian Statistical Institute, 203 B.T. Road, Kolkata - 700108, India }\\
		\emph{\small $^1$Email:bharatnishad.kanpu@gmail.com , $^2$Email: dmrai23@gmail.com}\\
		\emph{\small $^3$Email: akdas@isical.ac.in}}
	\date{}
	\maketitle

\abstract
\noindent {This article presents a class of  new relaxation modulus-based iterative methods to process the large and sparse implicit complementarity problem (ICP).  Using two positive diagonal matrices, we formulate a fixed-point equation and prove that it is equivalent to  ICP. Also, we provide sufficient  convergence conditions for the proposed methods when the system matrix is a $P$-matrix or an $H_+$-matrix.}
	\maketitle
	\section{Introduction}\label{sec1}
	The implicit complementarity problem (ICP) was introduced by Bensoussan et al. in 1973. For details, see \cite{7} and  \cite{8}.  The linear complementarity problem (LCP) is a special case of the ICP.  The ICP is used in many fields, such as engineering and economics, scientific computing, stochastic optimal control problems and convex cones. The ICP is discussed well in the literature. For s, see \cite{2}, \cite{3}, \cite{4}, \cite{5} and  \cite{6}.\\
	The large and sparse matrices are matrices that have a large number of rows and columns but a small number of non-zero elements. In other words, they are matrices where the majority of the elements are zero. Sparse matrices are commonly used to represent complex systems or large datasets in fields such as computer science, mathematics, physics, and engineering. The sparsity of the matrix means that it is not practical to store each element individually, and specialized data structures and algorithms must be used to efficiently store and manipulate the matrix.\\
	Given matrix $A \in \mathbb{R}^{n\times n} $ is a large and sparse and  vector $q \in \mathbb{R}^n$, the implicit complementarity problem denoted as ICP$(q, A, \zeta)$ is to find the solution vector $ z \in \mathbb{R}^n $ to the following system:
	\begin{equation}\label{eq1}
		Az+q\geq 0, ~~~z-\zeta(z)\geq 0, ~~~ (z-\zeta(z))^T(Az+q)=0
	\end{equation}
	where $\zeta(z)$ is a mapping from $\mathbb{R}^n$ into $\mathbb{R}^n$. When we put $\zeta(z)=0$ in Equation (\ref{eq1}), the ICP$(q, A, \zeta)$ reduces to the linear complementarity problem (LCP). For more s on LCP and its applications, see \cite{9}, \cite{10}, \cite{35}, \cite{36}, \cite{37}, \cite{38}, \cite{39}, \cite{44}, \cite{45}, \cite{41}, \cite{42}, \cite{46}, \cite{47}, \cite{neogy2011singular}, \cite{dutta2022column}, \cite{neogy2006some} and \cite{43}.\\
	In recent decades to solve complementarity problems, many methods have been proposed by considering some equivalent problems \cite{26}, \cite{27} and \cite{28}, such as fixed point approaches \cite{11}, projection-type methods \cite{12}, smooth and nonsmooth Newton methods \cite{13} and \cite{14},  matrix multisplitting methods \cite{17}, \cite{18}, \cite{19}, \cite{20}, \cite{21},\cite{34} and inexact alternating direction methods \cite{22} and  \cite{23}   and implicit complementarity problems \cite{52} and \cite{49}.\\
	In 2016, Hong and Li \cite{50}  demonstrated numerically the superiority of the modulus-based matrix splitting (MMS) method over the projection fixed-point method and the Newton method when solving the ICP$(q, A, \zeta)$.   Zheng and Vong  \cite{51} proposed a modified modulus based iteration method and subsequently established the convergence analysis in 2019. Furthermore,
	Li et.al. \cite{48} presented a  class of modified relaxation two sweep   modulus-based matrix splitting  iteration methods to solve the ICP$(q, A, \zeta)$ in 2022. 
Motivated by Li et al., for solving large and sparse ICP$(q, A, \zeta)$, we propose a class of new relaxation modulus-based iterative methods to reduce the number of iterations and CPU time and accelerate the convergence performance. \\
	This article is presented as follows: Some useful notations, definitions  and lemmas are given in Section 2 which are required for the remaining sections of this work. In Section 3, we  introduce  a family of   new relaxation modulus-based iterative methods  constructed using the  equivalent fixed-point form of the large and sparse ICP$(q, A, \zeta)$. In Section 4, we establish some convergence criteria for the proposed methods.  Section 5 contains the conclusion of the article.
	\section{Preliminaries}\label{Preli}
	In this section,  we introduce  some basic notations, definitions and lemmas that will be used throughout the article to examine the convergence analysis of the proposed methods and some existing iterative methods for solving implicit complementarity problems.\\
	The following is a list of related notations that are used for a given large and sparse  matrix $A$:
	\begin{itemize}
		\item Let  $A=({a}_{ij})\in \mathbb{R}^{n\times n}$ and  $B=({b}_{ij})\in \mathbb{R}^{n\times n}$.  We use $A \geq $ $(\textgreater)$ $ B$ to   denotes $a_{ij}\geq (\textgreater)$ $ b_{ij}$ $\forall$ $1 \leq  i,j \leq n$;
		\item     $(\star)^{T}$  denotes the transpose of the given  matrix or  vector;
		\item We use  $A=0 \in \mathbb{R}^{n \times n}$  to denotes  $ a_{ij}=0 ~\forall~ i,j$;
		\item $\lvert A\rvert=(\lvert {a}_{ij}\rvert)$  $\forall ~i,j$ and $A^{-1}$ represents the inverse of the matrix $A$;
		\item $\Omega_{1}$, $\Omega_{2}$ are real positive diagonal matrices of order $n$ and $\phi \in \mathbb{R}^{n \times n}$ is a relaxation  diagonal matrix.
		\item Let $x, y \in \mathbb{R}^n,$  min$(x,y)=$ the vector whose $i^{th}$ component is min$(x_{i}, y_{i})$ and $\|x\|_{2}=\sqrt{\sum_{i=1}^{n}}x^2_{i}$;\\
		\item Assume $A=D-L-U $ where $ D = diag(A)$ and  $L$, $U$ are the strictly lower, upper triangular 
		matrices of $A$, respectively;
		\item Let $A=M-N$ be a  splitting. We use $M\Omega_{1}=M_{\Omega_{1}}$, $N\Omega_{1}=N_{\Omega_{1}}$, $A\Omega_{1}=A_{\Omega_{1}}$, $B\Omega_{1}=B_{\Omega_{1}}$, $D\Omega_{1}=D_{\Omega_{1}}$  and $\phi\Omega_{1}=\phi_{\Omega_{1}}$.
	\end{itemize} 
	\begin{defn}\cite{31}
		Suppose  $ A=(a_{ij})\in  \mathbb{R}^{n\times n}$. Then its  comparison matrix $\langle A \rangle =(\langle a_{ij} \rangle)$  is defined by $\langle a_{ij} \rangle$ = $\lvert {a}_{ij}\rvert$ \text{if $i=j$} and $-\rvert {a}_{ij}\rvert$ \text{if $i\neq j$} for $i,j = 1,2,\ldots,n$.
	\end{defn}
	\begin{defn}\label{def1}\cite{32}
		Suppose  $ A \in \mathbb{R}^{n\times n}$. Then $A$ is said to be  a $Z$-matrix if all of its off-diagonal elements are nonpositive; an $M$-matrix if $A^{-1}\geq 0$ as well as   $Z$-matrix; an $H$-matrix; if   $\langle A \rangle$ is an $M$-matrix and an $H_+$-matrix if it is an $H$-matrix as well as  ${a}_{ii} ~\textgreater ~0$ for $~i =1,2,\ldots,n$.
	\end{defn} 
	\begin{defn}\cite{35}
		Let  $A\in \mathbb{R}^{n\times n}$. Then $A$ is said to be a $P$-matrix if all of its principle minors are positive i.e. det$({A}_{\alpha \alpha})~\textgreater ~0$  for all  $\alpha \subseteq \{1,2,\ldots, n\}$.
	\end{defn}
	\begin{defn}\label{def2}\cite{32}
		The splitting $A = M-N $ is called an  $M$-splitting if $M$ is a nonsingular $M$-matrix and $N \geq 0$; an $H$-splitting if $\langle M \rangle -\lvert N \rvert $ is an $M$-matrix.
	\end{defn} 
	\begin{defn}{\cite{54}}
		Let $A=(a_{ij})\in \mathbb{R}^{n \times n}$. Then $A $ is said to be a strictly diagonally dominant (sdd) matrix  if $\lvert a_{ii}\rvert\textgreater\sum_{j\neq i}\lvert a_{ij}\rvert$ for $i,j=1,2,\ldots,n.$
	\end{defn}
	\begin{lem}\label{lem1}\cite{32}
		Suppose  $A, {B} \in \mathbb{R}^{n\times n}$.
		\begin{enumerate}
			\item  If $A$ is an $M$-matrix, ${B}$ is a $Z$-matrix and $A \leq {B}$, then ${B}$ is an $M$-matrix. 
			\item If $A$ is an $H$-matrix, then $\lvert A^{-1}\rvert~\leq ~\langle A\rangle^{-1}$.
			\item  If $A \leq {B}$ then $\rho(A) ~\leq~  \rho({B})$  where $\rho(\star)$ represents the spectral radius of the matrix.
		\end{enumerate}
	\end{lem}
	\begin{lem}\label{lem4} \cite{32}
		Suppose  $0 \leq A \in \mathbb{R}^{n \times n}  $. If there exist $0 <u  \in \mathbb{R}^{n}$ and a scalar $\alpha ~\textgreater ~ 0$  such that $Au \leq  \alpha u$ then $\rho(A) \leq  \alpha $. Moreover, if~ $Au ~\textless~ \alpha u$ then $\rho(A) \textless \alpha$. 
	\end{lem}
	\begin{lem} \cite{54}\label{2.4}
		Let $A\in \mathbb{R}^{n \times n}$ be a sdd matrix. Then $$\|A^{-1}E\|_\infty \leq \max_{1\leq i \leq n}\frac{(\lvert E \rvert e)_{i}}{(\langle A \rangle e)_{i}},~~ \mbox{for all}~ E \in \mathbb{R}^{n \times n}$$ where $e=(1,1,\ldots,1)^T$.
	\end{lem}
	\begin{lem}\cite{53}
		Let $A \in \mathbb{R}^{n \times n}$ be  a nonsingular $M$-matrix then there exists a positive diagonal matrix $V$ such that $AV$ is an sdd matrix.
	\end{lem}
	Now, we introduce the modulus based matrix splitting iteration method for ICPs established by Hong and Li in \cite{50}. Define the modulus transformation 
	$$z-\zeta(z)=g(z)=\frac{1}{\gamma}(\lvert x \rvert +x)$$ and $$\omega=Az+q= \frac{\Omega}{\gamma}(\lvert x \rvert -x)$$
	where $\Omega$ is a positive diagonal matrix, $\gamma$ is a positive constant. If $\zeta(z)$ is invertible and $A=M-N$ is a splitting of the matrix $A$, ICP$(q, A, \zeta)$ can be transformed to an implicit fixed-point equation equivalently, which is described as follows.
	\begin{lem}\cite{48}
		For the $ICP(q, A, \zeta)$ , the following statements hold true:
	\noindent 	(1). If $z$ is the solution for the ICP(q,A, $\zeta$)
		, $x=\frac{\gamma}{2}(z-\Omega^{-1}\omega)-\zeta(z)$ satisfies:
		\begin{equation}\label{eq1'}
			(\Omega +Mx)=Nx+(\Omega-A)\lvert x \rvert -\gamma q- \gamma A \zeta[g^{-1}(\frac{1}{\gamma}(\lvert x \rvert +x))]
		\end{equation}
	\noindent 	(2) If $x$ satisfies Equation (\ref{eq1'}), then 	$z-\zeta(z)=g(z)=\frac{1}{\gamma}(\lvert x \rvert +x)$ and $\omega=Az+q= \frac{\Omega}{\gamma}(\lvert x \rvert -x)$, where $z$ is the solution for the ICP$(q, A, \zeta)$.
	\end{lem}
\noindent                                                                                   	However, there are some drawbacks in the modulus-based matrix splitting iteration method. For instance, the stopping criterion for inner iterations is not clear, which may cost extra elapsed CPU time. Moreover, the initial vector is difficult for  us to choose so that $z^{(0)}\in \{z \lvert  g(z) \geq 0, \omega \geq  0 \} $ is satisfied. For this, Zheng and Vong proposed a modified framework and
	gave an effective stopping criterion for improvement in \cite{51}. The algorithm removes the restriction on the initial vector
	and the whole process does not start from the initial solution vector $z^{(0)}$
	but from the unconditional initial modulus vector $x^{0}$.
	\begin{mthd}\cite{51} The modified modulus-based matrix splitting iteration method, abbreviated as MMS method. For  computing  $z^{(k+1)}\in \mathbb{R}^{n}$  use the following steps.\\
	\noindent	\textbf{Step 0}: Initialize a vector $x^{(0)} \in \mathbb{R}^{n}$,  $\epsilon ~\textgreater ~ 0 $  and   set $ k=0 $.\\ 
	\noindent	\textbf{Step 1} Compute $z^{(k)}$ by 
		\begin{equation}\label{000}
			Az^{(k)}=\frac{1}{\gamma}\Omega(\lvert x^{(k)} \rvert -x^{(k)})-q
		\end{equation}
		\textbf{Step 2} If $ Res(z^{(k)})=\|min\{Az^{(k)}+q,z^{(k)}-\zeta(z^{(k)})\}\|_{2} \textless $ $\epsilon$,  then stop.\\
	\noindent	\textbf{Step 3}: Using the following scheme, generate  the sequence $x^{(k)}$:
		\begin{equation*}
			\begin{split}
				\begin{split}
					x^{(k+1)} &=	(\Omega_{2}+M_{\Omega_{1}}+\phi_{\Omega_{1}})^{-1}[(N_{\Omega_{1}}+\phi_{\Omega_{1}})x^{(k)}+(\Omega_{2}-A_{\Omega_{1}})\lvert x^{(k)} \rvert\\&-\gamma A \zeta(z^{(k)})-\gamma q] 
				\end{split}
			\end{split}
		\end{equation*}
\noindent	\textbf{Step 4}:  Set $k=k+1$ and return to step 1.\\ 
	If the coefficient matrix $A$ of the linear system (\ref{000}) is large sparse, the Krylov subspace method can be applied.\\
	 
\end{mthd} 
	\begin{mthd} \cite{48} Nan li et.al. introduced a   modified relaxation two-sweep to the new scheme to speed up the performance of iteration
		methods.    For  computing  $z^{(k+1)}\in \mathbb{R}^{n}$  use the following steps.    \\
	\noindent 	\textbf{Step 0}: Initialize a vector $x^{(0)} \in \mathbb{R}^{n}$,  $\epsilon ~\textgreater ~ 0 $  and   set $ k=0 $.\\ 
	\noindent 	\textbf{Step 1} Compute $z^{(k)}$:
		$$Az^{(k)}=\frac{1}{\gamma}\Omega(\lvert x^{(k)} \rvert -x^{(k)})-q$$
	\noindent 	\textbf{Step 2} If $ Res(z^{(k)})=\|min\{Az^{(k)}+q,z^{(k)}-\zeta(z^{(k)})\}\|_{2} \textless $ $\epsilon$,  then stop.\\
	\noindent 	\textbf{Step 3}: Using the following scheme, generate  the sequence $x^{(k)}$:
		\begin{equation*}
			\begin{split}
				\begin{split}
					x^{(k+1)} &=	(\Omega_{2}+M_{\Omega_{1}}+\phi_{\Omega_{1}})^{-1}[(N_{\Omega_{1}}+\phi_{\Omega_{1}})x^{(k)}+(\Omega_{2}-A_{\Omega_{1}})\lvert x^{(k)} \rvert\\&-\gamma A \zeta(z^{(k)})-\gamma q] 
				\end{split}
			\end{split}
		\end{equation*}
	\noindent 	\textbf{Step 4}:  Set $k=k+1$ and return to step 1. 
\end{mthd}
	\noindent In the following section, we introduce a family of new relaxation modulus based iterative methods. These methods help us reduce the number of iterations and the
	time required by the CPU, which improves convergence performance.
	\section{Main results}\label{sec2}
	In this section, we present a class of  new relaxation modulus-based iterative methods for solving large and sparse ICP$(q, A,\zeta)$. First, we construct a new equivalent expression of the ICP$(q, A,\zeta)$, the s are as follows.
	\begin{thm}\label{thm0}
		Let $A \in \mathbb{R}^{n \times n}$ and $q \in \mathbb{R}^{n}$. Suppose $\Omega_{1}$, $\Omega_{2} \in \mathbb{R}^{n \times n}$  are two positive diagonal matrices and $\phi \in R^{n}$ is relaxation diagonal matrix. Then the ICP$(q, A,\zeta)$ is equivalent to 
		\begin{equation} \label{eq2}
			(\Omega_2+M_{\Omega_{1}}+\phi_{\Omega_{1}})x=(N_{\Omega_{1}}+\phi_{\Omega_{1}})x+(\Omega_{2}-A_{\Omega_{1}})\lvert x\rvert-\gamma A\zeta(z)-\gamma q.
		\end{equation}
	\end{thm}
	\begin{proof}
		Suppose  $z=\zeta(z)+\frac{\Omega_{1}}{\gamma}(\lvert x \rvert+x)$ and $\phi= \frac{1}{\gamma}\Omega_{2}(\lvert x \rvert-x)$. From Equation (\ref{eq1}), we have
		$$\frac{1}{\gamma}\Omega_{2}(\lvert x \rvert-x)=A(\zeta(z)+\frac{\Omega_{1}}{\gamma}(\lvert x \rvert+x))+q$$
		This implies that 
		\begin{equation*}
			(\Omega_{2}+A\Omega_{1}) x  =(\Omega_{2}-A\Omega_{1})\lvert x \rvert-\gamma A \zeta(z)-\gamma q 
		\end{equation*}
		Let $A= (M+\phi)-(N+\phi)$ be a splitting, where $\phi$ is a relaxation  diagonal matrix. Then the above equation can be rewritten as   
		\begin{equation*}
			(\Omega_{2}+M\Omega_{1}+\phi\Omega_{1})x =(N\Omega_{1}+\phi\Omega_{1})x+(\Omega_{2}-A\Omega_{1})\lvert x \rvert-\gamma A \zeta(z)-\gamma q 
		\end{equation*}
		This implies that
		\begin{equation*}
			(\Omega_{2}+M_{\Omega_{1}}+\phi_{\Omega_{1}})x =(N_{\Omega_{1}}+\phi_{\Omega_{1}})x+(\Omega_{2}-A_{\Omega_{1}})\lvert x \rvert-\gamma A \zeta(z)-\gamma q 
		\end{equation*}
	\end{proof}
	\noindent According to the statement in the Theorem \ref{thm0}, we are able to find the solution to ICP$(q, A,\zeta)$ with the use of an implicit fixed-point Equation  (\ref{eq2}). In order to solve the large and sparse ICP$(q, A,\zeta)$ from  Equation  ($\ref{eq2}$) we establish a  new relaxation modulus-based iterative method which is known as Method 3.1.
	\begin{mthd}\label{mthd1}
		Let $A=(M+\phi)-(N+\phi)$ be a splitting of the matrix $A\in \mathbb{R}^{n\times n}$ and $q\in \mathbb{R}^{n}$. Suppose that $\Omega_{2}+M_{\Omega_{1}}+\phi_{\Omega_{1}}$ is a nonsingular matrix. Then we use the following equation for  Method  \ref{mthd1} given as 
		\begin{equation}\label{eq4}
			\begin{split}
				x^{(k+1)} &=	(\Omega_{2}+M_{\Omega_{1}}+\phi_{\Omega_{1}})^{-1}[(N_{\Omega_{1}}+\phi_{\Omega_{1}})x^{(k)}+(\Omega_{2}-A_{\Omega_{1}})\lvert x^{(k)} \rvert-\gamma A \zeta(z^{(k)})\\&-\gamma q] 
			\end{split}
		\end{equation}
		Let residual be the Euclidean norm of the error vector which is defined in \cite{48} given as follows:  $$ Res(z^{(k)})=\|min(z^{(k)}-\zeta(z^k), Az^{(k)}+q) \|_{2}.$$ 
		Consider an   initial vector $x^{(0)}\in \mathbb{R}^n$. For $k=0,1,2,\ldots,$ the iterative process continues until the iterative sequence $\{z^{(k)}\}_{k=0}^{+\infty} \subset \mathbb{R}^n$ converges. The iterative process stops if $Res(z^{(k)})$ $\textless $ $ \epsilon $. For  computing  $z^{(k+1)}\in \mathbb{R}^{n}$ we use the following steps.\\
		\noindent 	\textbf{Step 0}: Initialize a vector $x^{(0)} \in \mathbb{R}^{n}$,  $\epsilon~ \textgreater~  0 $  and   set $ k=0 $.\\ 
		\noindent 	\textbf{Step 1} Compute $z^{(k)}$:
			$$Az^{(k)}=\frac{1}{\gamma}\Omega(\lvert x^{(k)} \rvert -x^{(k)})-q$$
		\noindent 	\textbf{Step 2} If $ Res(z^{(k)})=\|min\{Az^{(k)}+q,z^{(k)}-\zeta(z^{(k)})\}\|_{2} \textless $ $\epsilon$,  then stop.\\
		\noindent 	\textbf{Step 3}: Using the following scheme, generate  the sequence $x^{(k)}$:
			\begin{equation*}
				\begin{split}
					\begin{split}
						x^{(k+1)} &=	(\Omega_{2}+M_{\Omega_{1}}+\phi_{\Omega_{1}})^{-1}[(N_{\Omega_{1}}+\phi_{\Omega_{1}})x^{(k)}+(\Omega_{2}-A_{\Omega_{1}})\lvert x^{(k)} \rvert\\&-\gamma A \zeta(z^{(k)})-\gamma q] 
					\end{split}
				\end{split}
			\end{equation*}
		\noindent 	\textbf{Step 4}:   Set $k=k+1$ and return to step 1. 
	\end{mthd}
	\noindent Moreover, Method \ref{mthd1}  provides a general framework to  solve the large and sparse  ICP$(q, A, \zeta)$.  We obtain  a class of new relaxation  modulus-based  iterative methods.
	\begin{enumerate}
		\item when $M=A,~ \Omega_1 = I, ~\Omega_2 =I,~ N=0 $, from Equation  ($\ref{eq4}$) we obtain 
		\begin{equation*}
			(I+M +\phi)x^{(k+1)}=\phi x^{(k)}+(I-A)\lvert x^{(k)}\rvert-\gamma q\zeta(z^{(k)})-\gamma q,	
		\end{equation*}
		known as  ``new relaxation modulus-based iterative (NRM) method$"$
		\item when $M=A, ~\Omega_1 =  I, ~\Omega_2 =\alpha I, ~N=0 $, from Equation  ($\ref{eq4}$) we obtain 
		\begin{equation*}
			(\alpha I+M +\phi)x^{(k+1)}=\phi x^{(k)}+(\alpha I-A)\lvert x^{(k)}\rvert-\gamma q\zeta(z^{(k)})-\gamma q,
		\end{equation*}
		known as  ``new relaxation Modified  modulus-based iterative (NRMM) method$"$
		\item when $M=\frac{1}{\alpha}(D-\beta L), ~ N=\frac{1}{\alpha}[(1-\alpha)D+(\alpha-\beta)L+\alpha U]$,  from Equation  ($\ref{eq4}$) we obtain 
		\begin{equation*}
			\begin{split}
				(\alpha \Omega_{2}+(D_{\Omega_{1}}-\beta L_{\Omega_{1}})+\alpha \phi_{\Omega_{1}})x^{(k+1)}&=((1-\alpha)D_{\Omega_{1}}+(\alpha-\beta)L_{\Omega_{1}}\\&+\alpha U_{\Omega_{1}}+\alpha \phi_{\Omega_{1}})x^{(k)}+\alpha(\lvert(\Omega_{2}\\&-A_{\Omega_{1}})\lvert x^{(k)}\rvert -\alpha \gamma A \zeta(z^{(k)})-\gamma \alpha q,
			\end{split}
		\end{equation*}
	\end{enumerate}
	known as  ``new relaxation modulus-based  accelerated over relaxation iterative (NRMAOR) method$"$. \\
	The NRMAOR method clearly converts into the following methods.
	\begin{enumerate}
		\item new relaxation modulus-based  successive overrelaxation iterative (NRMSOR) method when $(\alpha, \beta)$ takes the values $(\alpha, \beta)$.
		\item new relaxation modulus-based  Gauss-Seidel iterative (NRMGS) method when $(\alpha, \beta)$ takes the values $(1, 1)$.
		\item  iterative modulus-based Jacobi iterative (NRMJ) method when $(\alpha, \beta)$ takes the values $(1, 0)$.
	\end{enumerate}
	\section{Convergence analysis}
	Some conditions for the unique solution of the  ICP$(q, A, \zeta)$ are presented in \cite{1} and \cite{2}. Without losing generality, we consider in this article that the ICP$(q, A, \zeta)$ has a unique solution.\\
	In this part of the article, we discuss the convergence analysis for the Method \ref{mthd1}. To begin, we discuss the convergence condition assuming that the system matrix $A$ is a $P$-matrix.
	\begin{thm}\label{thm1}
		Let $A=(M+\phi)-(N+\phi)$ be a splitting of the $P$-matrix $A\in \mathbb{R}^{n\times n}$ and   $(\Omega_{2}+M_{\Omega_{1}}+\phi_{\Omega_{1}})$ be a  nonsingular matrix. If $\rho(L) \textless 1$ where $L=\lvert(\Omega_{2}+M_{\Omega_{1}}+\phi_{\Omega_{1}})^{-1}\rvert(\lvert N_{\Omega_{1}}+\phi_{\Omega_{1}}\rvert+\lvert \Omega_{2}-A_{\Omega_{1}}\rvert+2 \lvert A \rvert \Psi \lvert A^{-1} \rvert \Omega_{2} )$ and $\lvert \zeta(x)-\zeta(y)\rvert \leq \Psi \lvert x-y \rvert$ $\forall ~x,y \in \mathbb{R}^{n}$ where $\Psi$ is a nonnegative   matrix of order $n$.
		Then  the iterative sequence $\{z^{(k)}\}_ {k=0}^{+\infty} \subset \mathbb{R}^n$ generated by Method $\ref{mthd1}$ converges to a unique solution $z^*\in \mathbb{R}^n$  for any  initial  vector $x^{(0)}\in \mathbb{R}^n$.
	\end{thm}
	\begin{proof}
		Let $x^{*}$ be a solution of the ICP$(q, A,\zeta)$.  From Equation  ($\ref{eq2}$) we obtain
		\begin{equation}\label{eq5}
			\begin{split}
				(\Omega_{2}+M_{\Omega_{1}}+\phi_{\Omega_{1}})x^*&=(N_{\Omega_{1}}+\phi_{\Omega_{1}})x^*+ (\Omega_{2}-A_{\Omega_{1}})\lvert x^*
				\rvert -\gamma A \zeta(z^{*})-\gamma q.
			\end{split}
		\end{equation}
		From Equation  ($\ref{eq4}$) and Equation  ($\ref{eq5}$), we obtain
		\begin{equation*}\label{eq4}
			\begin{split}
				(\Omega_{2}+M_{\Omega_{1}}+\phi_{\Omega_{1}})(x^{(k+1)}-x^{*})&=(N_{\Omega_{1}}+\phi_{\Omega_{1}})(x^{(k)}-x^{*})+(\Omega_{2}-A_{\Omega_{1}})\lvert x^{(k)}\rvert- (\Omega_{2}\\&-A_{\Omega_{1}})\lvert x^*
				\rvert +\gamma A \zeta(z^{*})-\gamma A \zeta(z^{(k)}).
			\end{split}
		\end{equation*}
		This implies that
		\begin{equation*}\label{eq4}
			\begin{split}
				(x^{(k+1)}-x^{*})&=(\Omega_{2}+M_{\Omega_{1}}+\phi_{\Omega_{1}})^{-1}[(N_{\Omega_{1}}+\phi_{\Omega_{1}})(x^{(k)}-x^{*})+(\Omega_{2}-A_{\Omega_{1}})(\lvert x^{(k)}\rvert\\&-\lvert x^{*}\rvert)+\gamma A (\zeta(z^{*})- \zeta(z^{(k)})].
			\end{split}
		\end{equation*}
		By applying   both side modulus, it follows that 
		\begin{equation*}
			\begin{split}	\lvert x^{(k+1)}-x^{*}\rvert&=\lvert(\Omega_{2}+M_{\Omega_{1}}+\phi_{\Omega_{1}})^{-1}[(N_{\Omega_{1}}+\phi_{\Omega_{1}})(x^{(k)}-x^{*})+(\Omega_{2}-A_{\Omega_{1}})(\lvert x^{(k)}\rvert\\&-\lvert x^{*}\rvert)+\gamma A (\zeta(z^{*})- \zeta(z^{(k)})]\rvert\\
				&\leq \lvert(\Omega_{2}+M_{\Omega_{1}}+\phi_{\Omega_{1}})^{-1}\rvert[\lvert(N_{\Omega_{1}}+\phi_{\Omega_{1}})(x^{(k)}-x^{*})\rvert+\lvert(\Omega_{2}-A_{\Omega_{1}})(\lvert x^{(k)}\rvert\\&-\lvert x^{*}\rvert)\rvert+\gamma\lvert A\rvert \lvert (\zeta(z^{*})- \zeta(z^{(k)})\rvert] \\
				&\leq \lvert(\Omega_{2}+M_{\Omega_{1}}+\phi_{\Omega_{1}})^{-1}\rvert[\lvert(N_{\Omega_{1}}+\phi_{\Omega_{1}})\rvert \lvert x^{(k)}-x^{*}\rvert+\lvert\Omega_{2}-A_{\Omega_{1}}\rvert \lvert\lvert x^{(k)}\rvert\\&-\lvert x^{*}\rvert\rvert+\gamma\lvert A\rvert \lvert \zeta(z^{*})- \zeta(z^{(k)})\rvert]
			\end{split}
		\end{equation*}
		Since $\rvert \lvert x \rvert- \lvert y\rvert\rvert \leq \lvert x- y \rvert$,  the above inequality  can be written as  
		\begin{equation*}
			\begin{split}
				&\leq \lvert(\Omega_{2}+M_{\Omega_{1}}+\phi_{\Omega_{1}})^{-1}\rvert[\lvert N_{\Omega_{1}}+\phi_{\Omega_{1}}\rvert \lvert x^{(k)}-x^{*}\rvert+\lvert\Omega_{2}-A_{\Omega_{1}}\rvert \lvert x^{(k)}-x^{*}\rvert \\&+\gamma\lvert A\rvert \lvert (\zeta(z^{*})- \zeta(z^{(k)})\rvert]
			\end{split}
		\end{equation*}
		We have $\lvert (\zeta(z^{(k)})- \zeta(z^{*})\rvert \leq \Psi\lvert z^{(k)}- z^{*}\rvert \leq \frac{2}{\gamma}\Psi \lvert A^{-1}\rvert \Omega_{2}\lvert x^{(k)}- x^{*}\rvert  $
		\begin{equation*}
			\begin{split} &\leq \lvert(\Omega_{2}+M_{\Omega_{1}}+\phi_{\Omega_{1}})^{-1}\rvert[\lvert N_{\Omega_{1}}+\phi_{\Omega_{1}}\rvert \lvert x^{(k)}-x^{*} \rvert+\lvert\Omega_{2}-A_{\Omega_{1}}\rvert \lvert x^{(k)}-x^{*}\rvert \\&+2\lvert A\rvert \Psi \lvert A^{-1}\rvert \Omega_{2}\lvert x^{(k)}- x^{*}\rvert]
				\\&\leq \lvert(\Omega_{2}+M_{\Omega_{1}}+\phi_{\Omega_{1}})^{-1}\rvert[\lvert(N_{\Omega_{1}}+\phi_{\Omega_{1}})\rvert +\lvert\Omega_{2}-A_{\Omega_{1}}\rvert  +2\lvert A\rvert \Psi \lvert A^{-1}\rvert \Omega_{2}]\lvert x^{(k)}- x^{*}\rvert
				\\&\leq T \lvert x^{(k)}-x^*\rvert,
			\end{split}
		\end{equation*}	
		where $T=\lvert(\Omega_{2}+M_{\Omega_{1}}+\phi_{\Omega_{1}})^{-1}\rvert[\lvert(N_{\Omega_{1}}+\phi_{\Omega_{1}})\rvert +\lvert\Omega_{2}-A_{\Omega_{1}}\rvert  +2\lvert A\rvert \Psi \lvert A^{-1}\rvert \Omega_{2}].$
		If $ \rho(T) \textless 1$, then
		$$\lvert x^{(k+1)}-x^*\rvert \textless  \lvert x^{(k)}-x^*\rvert.$$
		Hence, for any  initial vector $x^{(0)}\in \mathbb{R}^n$ the iterative sequence $\{z^{(k)}\}_{k=0}^{+\infty} \subset \mathbb{R}^n$ generated by Method $\ref{mthd1}$ converges to a unique solution $z^{*} \in \mathbb{R}^n$.
	\end{proof}
	\noindent In the  proof of the Theorem \ref{thm1}, we note that 
	\begin{equation*}
		\begin{split}
			\lvert\Omega_{2}- A_{\Omega_{1}}\rvert=\lvert A_{\Omega_{1}} - \Omega_{2}
			\rvert&=\lvert(M_{\Omega_{1}}+\phi_{\Omega_{1}})-(N_{\Omega_{1}}+\phi_{\Omega_{1}})-\Omega_{2}\rvert\\ &
			= \lvert(M_{\Omega_{1}}+\phi_{\Omega_{1}}-\Omega_{2})-(N_{\Omega_{1}}+\phi_{\Omega_{1}})\rvert\\&
			\leq \lvert M_{\Omega_{1}}+\phi_{\Omega_{1}}-\Omega_{2}\rvert+\lvert N_{\Omega_{1}}+\phi_{\Omega_{1}}\rvert.
		\end{split}
	\end{equation*}
	Based on above inequality we can obtain Corollary 4.1.
	\begin{coro}
		Let $A=(M+\phi)-(N+\phi)$ be a splitting of a $P$-matrix $A\in \mathbb{R}^{n\times n}$ and  $M_{\Omega_{1}}+\phi_{\Omega_{1}}+\Omega_{2}$ be nonsingular matrix. If $\rho(\bar{L}) \textless 1$ where $\bar{L}=\lvert(\Omega_{2}+M_{\Omega_{1}}+\phi_{\Omega_{1}})^{-1}\rvert(2\lvert N_{\Omega_{1}}+\phi_{\Omega_{1}}\rvert+\lvert M_{\Omega_{1}}+\phi_{\Omega_{1}}-\Omega_{2}\rvert+2 \lvert A 
		\rvert \Psi \lvert A^{-1} 
		\rvert \Omega_{2}),$
		then for any  initial vector $x^{(0)}\in \mathbb{R}^n$, the iterative sequence $\{z^{(k)}\}_{k=0}^{+\infty} \subset \mathbb{R}^n$  generated by Method $\ref{mthd1}$ converges to a unique solution $z^{*} \in \mathbb{R}^n$.
	\end{coro}
	\noindent	We can easily obtain corollary 4.2 from the proof of the Theorem $\ref{thm1}$ by using $\|\cdot\|_2$.
	\begin{coro}
		Let $A=(M+\phi)-(N+\phi)$ be a splitting of the $P$-matrix $A\in \mathbb{R}^{n\times n}$ and  $M_{\Omega_{1}}+\phi_{\Omega_{1}}+\Omega_{2}$ be nonsingular  matrix.\\	
		Let $${S}=\|(\Omega_{2}+M_{\Omega_{1}}+\phi_{\Omega_{1}})^{-1}\|_{2}(\|N_{\Omega_{1}}+\phi_{\Omega_{1}}\|_{2}+\|A_{\Omega_{1}}-\Omega_{2}\|_{2}+2 \lvert A 
		\rvert \Psi \lvert A^{-1} 
		\rvert \Omega_{2}\|_{2}),$$
		and
		$$\bar{S}=\|(\Omega_{2}+M_{\Omega_{1}}+\phi_{\Omega_{1}})^{-1}\|_{2}(2\|N_{\Omega_{1}}+\phi_{\Omega_{1}}\|_{2}+\| M_{\Omega_{1}}+\phi_{\Omega_{1}}-\Omega_{2}\|_{2}+2 \lvert A 
		\rvert \Psi \lvert A^{-1} 
		\rvert \Omega_{2}\|_{2}).$$
		If $\rho({S}) \textless 1$ or $\rho(\bar{S}) \textless 1$,  
		then for any  initial vector $x^{(0)}\in \mathbb{R}^n$, the iterative sequence $\{z^{(k)}\}_{k=0}^{+\infty} \subset \mathbb{R}^n$  generated by  Method $\ref{mthd1}$ converges to a unique solution $z^{*} \in \mathbb{R}^n$.
	\end{coro}
	\noindent	In the following results, we discuss some convergence conditions for Method \ref{mthd1} when system matrix $A$ is an $H_+$-matrix.
	\begin{thm}\label{thm2}
		Let $A=(M+\phi)-(N+\phi)$ be an $H$-splitting of the $H_{+}$-matrix $A\in \mathbb{R}^{n\times n}$,   $\Omega_{2}\geq D_{\Omega_{1}}$ and $\langle M_{\Omega_{1}}\rangle-\lvert N_{\Omega_{1}}\rvert-2 \lvert A 
		\rvert \Psi \lvert A^{-1} 
		\rvert \Omega_{2}$ be an $M$-matrix. Let $\phi=(\phi_{ij})$, $\Omega_{1}=(\omega^{1}_{ij})$, $\Omega_2=(\omega^{2}_{ij})$, $M=(m_{ij})$ and $N=(n_{ij})$, where $i,j=1,2,\ldots,n$.   Then the iterative sequence $\{z^{(k)}\}_{k=0}^{+\infty} \subset \mathbb{R}^n$ generated by Method $\ref{mthd1}$ converges to a unique solution $z^{*} \in \mathbb{R}^n$ for any  initial vector $x^{(0)}\in \mathbb{R}^n$.
	\end{thm}	
	\begin{proof}
		Since  $\langle M+\phi\rangle-\lvert N+\phi\rvert$  is an $M$-matrix and 
		$ \langle M+\phi\rangle-\lvert N+\phi\rvert \leq  \langle M+\phi\rangle, $	
		\noindent   $\langle M+\phi\rangle$ is an $M$-matrix from Lemma $\ref{lem1}$. Since $\Omega_{2}\geq D_{\Omega_{1}}$, therefore we can obtain that $\Omega_{2}+M_{\Omega_{1}}+\phi_{\Omega_{1}}$ is an $H_{+}$-matrix and $$\lvert(\Omega_{2}+M_{\Omega_{1}}+\phi_{\Omega_{1}})^{-1}\rvert\leq (\Omega_{2}+\langle M_{\Omega_{1}}\rangle+\lvert\phi_{\Omega_{1}}\rvert)^{-1}.$$
		From Theorem $\ref{thm1}$, we obtain
		\begin{equation*}
			\begin{split}
				T&=\lvert(\Omega_{2}+M_{\Omega_{1}}+\lvert\phi_{\Omega_{1}}\rvert)^{-1}\rvert(\vert N_{\Omega_{1}}+\phi_{\Omega_{1}}\rvert+\lvert \Omega_{2}-A_{\Omega_{1}}\rvert+2 \lvert A 
				\rvert \Psi \lvert A^{-1} 
				\rvert \Omega_{2})\\&
				\leq (\Omega_{2}+\langle M_{\Omega_{1}}\rangle+\lvert\phi_{\Omega_{1}}\rvert)^{-1}(\lvert N_{\Omega_{1}}+\lvert\phi_{\Omega_{1}}\rvert\rvert+\lvert \Omega_{2}-A_{\Omega_{1}}\rvert+2 \lvert A 
				\rvert \Psi \lvert A^{-1} 
				\rvert \Omega_{2})\\&
				\leq (\Omega_{2}+\langle M_{\Omega_{1}}\rangle+\lvert\phi_{\Omega_{1}}\rvert)^{-1}(\langle M_{\Omega_{1}}\rangle+\lvert\phi_{\Omega_{1}}\rvert+\Omega_{2}-(\langle M_{\Omega_{1}}\rangle+\lvert\phi_{\Omega_{1}}\rvert+\Omega_{2})\\&+\lvert N_{\Omega_{1}}+\phi_{\Omega_{1}}\rvert+\lvert \Omega_{2}-D_{\Omega_{1}}\rvert+\lvert B_{\Omega_{1}}\rvert+2 \lvert A 
				\rvert \Psi \lvert A^{-1} 
				\rvert \Omega_{2})\\&
				\leq I- (\Omega_{2}+\langle M_{\Omega_{1}}\rangle+\lvert\phi_{\Omega_{1}}\rvert)^{-1}((\langle M_{\Omega_{1}}\rangle+\lvert\phi_{\Omega_{1}}\rvert+\Omega_{2})\\&-\lvert N_{\Omega_{1}}+\phi_{\Omega_{1}}\rvert-\lvert \Omega_{2}-D_{\Omega_{1}}\rvert-\lvert B_{\Omega_{1}}\rvert-2 \lvert A 
				\rvert \Psi \lvert A^{-1} 
				\rvert \Omega_{2})\\&
				\leq I- (\Omega_{2}+\langle M_{\Omega_{1}}\rangle+\lvert\phi_{\Omega_{1}}\rvert)^{-1}((\langle M_{\Omega_{1}}\rangle+\lvert\phi_{\Omega_{1}}\rvert)\\&-\lvert N_{\Omega_{1}}+\phi_{\Omega_{1}}\rvert+D_{\Omega_{1}}-\lvert B_{\Omega_{1}}\rvert-2 \lvert A 
				\rvert \Psi \lvert A^{-1} 
				\rvert \Omega_{2})\\
				&\leq I- (\Omega_{2}+\langle M_{\Omega_{1}}\rangle+\lvert\phi_{\Omega_{1}}\rvert)^{-1}(\langle M_{\Omega_{1}}+\phi_{\Omega_{1}}\rangle-\lvert N_{\Omega_{1}}+\phi_{\Omega_{1}}\rvert\\&+D_{\Omega_{1}}-\lvert B_{\Omega_{1}}\rvert-2 \lvert A 
				\rvert \Psi \lvert A^{-1} 
				\rvert \Omega_{2})\\&
				\leq I- (\Omega_{2}+\langle M_{\Omega_{1}}\rangle+\phi_{\Omega_{1}})^{-1}(\langle M_{\Omega_{1}}\rangle-\lvert N_{\Omega_{1}}\rvert+D_{\Omega_{1}}-\lvert B_{\Omega_{1}}\rvert-2 \lvert A 
				\rvert \Psi \lvert A^{-1} 
				\rvert \Omega_{2}).
			\end{split}
		\end{equation*}
		\noindent Let $\lvert A 
		\rvert \Psi \lvert A^{-1} 
		\rvert \Omega_{2}= G=(g_{ij})~ \textgreater~ 0 ~\forall~ i,j= 1,2,\ldots,n.$\\
		If $\langle M_{\Omega_{1}}\rangle-\lvert N_{\Omega_{1}}\rvert-2 \lvert A 
		\rvert \Psi \lvert A^{-1} 
		\rvert \Omega_{2}$
		is an $M$-matrix, there exist a positive vector $v$ such that $$(\lvert m_{ii}\rvert\omega_{ii}-\lvert n_{ii}\rvert\omega_{ii}-2g_{ii})v_{i}-\sum_{j\neq i}(\lvert m_{ij}\rvert\omega_{ij}+\lvert n_{ij}\rvert\omega_{ij}+2g_{ij})v_{j}~\textgreater~ 0.$$
		It holds that 
		$$a_{ii}\omega_{ii}=\lvert a_{ii}\omega_{ii} \rvert \geq \lvert m_{ii}\omega_{ii}- n_{ii\omega_{ii}}\rvert \geq \lvert m_{ij}\omega_{ii}\rvert -\lvert n_{ij}\omega_{ii}\rvert -2g_{ij}$$ 
		and                                       
		$$\lvert a_{ii}w_{ii} \rvert \leq  \lvert m_{ij}\omega_{ii}\rvert +\lvert n_{ij}\omega_{ii}\rvert -2g_{ij}~~ \forall ~i=1,2,\ldots,n.$$
		Therefore 
		$$((a_{ii}w_{ii})v_{i}- \sum_{j\neq i}\lvert a_{ij}w_{ii}\rvert v_{j})= (D_{\Omega_{1}}-\lvert B_{\Omega_{1}} \rvert) ~\textgreater~ 0~ \forall~ i=1,2,\ldots,n. $$
		Therefore, we obtain that 
		\begin{eqnarray*}
			(\langle M_{\Omega_{1}}\rangle-\lvert N_{\Omega_{1}}\rvert+D_{\Omega_{1}}-\lvert B_{\Omega_{1}}\rvert-2 \lvert A 
			\rvert \Psi \lvert A^{-1} 
			\rvert \Omega_{2})v~\textgreater~ 0.
		\end{eqnarray*}
		Therefore, 
		\begin{equation*}
			Tv\leq Iv- (\Omega_{2}+\langle M_{\Omega_{1}}\rangle+\phi_{\Omega_{1}})^{-1}(\langle M_{\Omega_{1}}\rangle-\lvert N_{\Omega_{1}}\rvert+D_{\Omega_{1}}-\lvert B_{\Omega_{1}}\rvert-2 \lvert A 
			\rvert \Psi \lvert A^{-1} 
			\rvert \Omega_{2})v ~\textless~ v.
		\end{equation*}
		By using the Lemma $\ref{lem4}$, we are able to determine $\rho(T)\textless 1$. According to Theorem {\ref{thm1}}, the iterative sequence $\{z^{(k)}\}_{k=0}^{+\infty} \subset R^n$ generated by algorithm $\ref{mthd1}$ converges to the unique solution $z^{*} \in R^n$ for any  initial vector $x^{(0)}\in R^n$.
	\end{proof}
	\begin{thm}\label{thm3}
		Let $A=(M+\phi)-(N+\phi)= D-L-U =D-B$ be a splitting of the $H_{+}$-matrix $A\in \mathbb{R}^{n\times n}$ where $B=L+U$  and $\rho(D_{\Omega_{1}}^{-1}\lvert B_{\Omega_{1}}\rvert+ \lvert A 
		\rvert \Psi \lvert A^{-1} 
		\rvert \Omega_{2}) \textless 1$. Let  $\Omega_{2}\geq D_{\Omega_{1}}$. If the parameters $\alpha$, $\beta$  satisfy 
		\begin{equation}\label{eq7}
			0 ~\leq~ \max\{\alpha, \beta\}\rho(D_{\Omega_{1}}^{-1}\lvert B_{\Omega_{1}}\rvert+ \lvert A 
			\rvert \Psi \lvert A^{-1} 
			\rvert \Omega_{2}) \textless \min\{1, \alpha\},
		\end{equation} 
		then the iterative sequence $\{z^{(k)}\}_ {k=0}^{+\infty} \subset \mathbb{R}^n$ generated by NRMAOR method converges to a unique solution $z^*\in \mathbb{R}^n$ for  any  initial vector $x^{(0)}\in \mathbb{R}^n$.
	\end{thm}
	\begin{proof} From the proof of the Theorem $\ref{thm2}$, we have \\
		$$\bar{H}-\bar{F}=\langle M_{\Omega_{1}}\rangle+\lvert\phi_{\Omega_{1}}\rvert+\Omega_{2}-\lvert N_{\Omega_{1}}+\phi_{\Omega_{1}}\rvert-\lvert \Omega_{2}-A_{\Omega_{1}}\rvert-2 \lvert A 
		\rvert \Psi \lvert A^{-1} 
		\rvert \Omega_{2}.$$
		Let $M=\frac{1}{\alpha}(D-\beta L)$  and $N=\frac{1}{\alpha}((1-\alpha)D+(\alpha-\beta)L+\alpha U).$ Then
		\begin{equation*}
			\begin{split}
				\bar{H}&=\langle M_{\Omega_{1}}\rangle+\lvert\phi_{\Omega_{1}}\rvert+\Omega_{2}\\&
				=\langle \frac{1}{\alpha}(D_{\Omega_1}-\beta L_{\Omega_1})\rangle+\lvert\phi_{\Omega_{1}}\rvert+\Omega_{2}\\&
				= \frac{1}{\alpha}[(D_{\Omega_1}-\beta \lvert L_{\Omega_1}\rvert)+\alpha\lvert\phi_{\Omega_{1}}\rvert+\alpha\Omega_{2}]
			\end{split}
		\end{equation*}
		and
		\begin{equation*}
			\begin{split}
				\bar{F}=\frac{1}{\alpha}[\lvert(1-\alpha)D_{\Omega_1}+(\alpha-\beta)L_{\Omega_1}+\alpha U_{\Omega_1}+\alpha\phi_{\Omega_{1}}\rvert+\alpha\lvert \Omega_{2}-A_{\Omega_{1}}\rvert+2\alpha \lvert A 
				\rvert \Psi \lvert A^{-1} 
				\rvert \Omega_{2}].
			\end{split}
		\end{equation*}
		Now we compute
		\begin{equation*}
			\begin{split}
				\alpha(\bar{H}-\bar{F})&=[(D_{\Omega_1}-\beta \lvert L_{\Omega_1}\rvert)+\alpha\lvert\phi_{\Omega_{1}}\rvert+\alpha\Omega_{2}]-[\lvert(1-\alpha)D_{\Omega_1}\\&+(\alpha-\beta)L_{\Omega_1}+\alpha U_{\Omega_1}+\alpha\phi_{\Omega_{1}}\rvert+\alpha\lvert \Omega_{2}-A_{\Omega_{1}}\rvert+2\alpha \lvert A 
				\rvert \Psi \lvert A^{-1} 
				\rvert \Omega_{2}]\\&
				= (D_{\Omega_1}-\beta \lvert L_{\Omega_1}\rvert)+\alpha\lvert\phi_{\Omega_{1}}\rvert+\alpha\Omega_{2}-\lvert(1-\alpha)D+\alpha\phi_{\Omega_{1}}\rvert \\&-\lvert\alpha L_{\Omega_1}+\alpha U-\beta L\rvert-\lvert \Omega_{2}-A_{\Omega_{1}}\rvert-2\alpha \lvert A 
				\rvert \Psi \lvert A^{-1} 
				\rvert \Omega_{2}\\&
				= (D_{\Omega_1}-\beta \lvert L_{\Omega_1}\rvert)+\alpha\lvert\phi_{\Omega_{1}}\rvert-\lvert(1-\alpha)D+\alpha\phi_{\Omega_{1}}\rvert \\&-\lvert\alpha B_{\Omega_{1}}-\beta L\rvert+\alpha D_{\Omega_{1}}-\alpha\lvert B_{\Omega_{1}}\rvert-2\alpha \lvert A 
				\rvert \Psi \lvert A^{-1} 
				\rvert \Omega_{2}\\&
				\geq  (1+\alpha-\lvert1-\alpha\rvert)D_{\Omega_{1}} -\lvert\alpha B_{\Omega_{1}}-\beta L\rvert\\&-\alpha\lvert B_{\Omega_{1}}\rvert-\beta \lvert L_{\Omega_1}\rvert-2\alpha \lvert A 
				\rvert \Psi \lvert A^{-1} 
				\rvert \Omega_{2}.
			\end{split}
		\end{equation*}
		Since  
		\begin{equation*}
			(1 +\alpha -\lvert1-\alpha \rvert)= 2\min\{1, \alpha\}
		\end{equation*} 
		and 
		\begin{equation*}
			\begin{split}
				\lvert\alpha B_{\Omega_{1}}-\beta L_{\Omega_1}\rvert+\alpha\lvert B_{\Omega_{1}}\rvert+\beta \lvert L_{\Omega_1}\rvert
				&\leq \alpha\lvert B_{\Omega_{1}}\rvert+\beta \lvert L_{\Omega_1}\rvert+\alpha\lvert B_{\Omega_{1}}\rvert+\beta \lvert L_{\Omega_1}\rvert
				\\&\leq 2\max\{\alpha,\beta\}\lvert B_{\Omega_{1}}\rvert.
			\end{split}
		\end{equation*}
		Therefore, \begin{equation*}
			\begin{split}
				\alpha(\bar{H}-\bar{F})&\geq 2\min\{1, \alpha\}D_{\Omega_1}-2\max\{\alpha,\beta\}\lvert B_{\Omega_{1}}\rvert-2\alpha \lvert A 
				\rvert \Psi \lvert A^{-1} 
				\rvert \Omega_{2}\\&
				\geq 2\min\{1, \alpha\}D_{\Omega_1}-2\max\{\alpha,\beta\}\lvert B_{\Omega_{1}}\rvert-2\max\{\alpha,\beta\} \lvert A 
				\rvert \Psi \lvert A^{-1} 
				\rvert \Omega_{2}
				\\&
				\geq 2\min\{1, \alpha\}D_{\Omega_1}-2\max\{\alpha,\beta\}(\lvert B_{\Omega_{1}}\rvert+2 \lvert A 
				\rvert \Psi \lvert A^{-1} 
				\rvert \Omega_{2})
				\\&
				\geq 2D_{\Omega_1}(\min\{1, \alpha\}I-\max\{\alpha,\beta\}D^{-1}_{\Omega_1}(\lvert B_{\Omega_{1}}\rvert+ \lvert A 
				\rvert \Psi \lvert A^{-1} 
				\rvert \Omega_{2}).
			\end{split}
		\end{equation*}
		Since $\rho(D^{-1}_{\Omega_1}(\lvert B_{\Omega_{1}}\rvert+ \lvert A 
		\rvert \Psi \lvert A^{-1} 
		\rvert \Omega_{2})\textless 1,$ it is easy to verify that $\alpha(\bar{H}-\bar{F})$ an $M$-matrix. Hence, we obtain that $\bar{H}-\bar{F}$ is an $M$-matrix.
	\end{proof}
	\begin{thm}
		Let $A=(M+\phi)-(N+\phi)$ be an $H$-splitting with an $H_+$-matrix $A \in \mathbb{R}^{n\times n}$.  Suppose $\Psi$ be a nonnegative matrix, there exists a positive diagonal matrix $V$ such that $(\langle M+\phi\rangle -(N+\phi))V$  is a strictly diagonally dominant matrix and one of the following conditions holds:\\
		(1) for  $\Omega_{2} \geq D_{\Omega_{1}}$,
		$(\langle M_{\Omega_{1}}+\phi_{\Omega_{1}}\rangle -\lvert N_{\Omega_{1}}+\phi_{\Omega_{1}}\rvert -\lvert A 
		\rvert \Psi \lvert A^{-1} 
		\rvert \Omega_{2})Ve ~\textgreater ~0;$\\
		(2) for $0 ~\textless~ \Omega_{2}~\textless~ D_{\Omega_{1}}$, $(\Omega_{2}-D_{\Omega_{1}}+\langle M_{\Omega_{1}}+\phi_{\Omega_{1}}\rangle -\lvert N_{\Omega_{1}}+\phi_{\Omega_{1}}\rvert -\lvert A 
		\rvert \Psi \lvert A^{-1} 
		\rvert \Omega_{2})Ve~ \textgreater ~0.$ \\
		Then the iterative sequence $\{z^{(k)}\}_ {k=0}^{+\infty} \subset \mathbb{R}^n$ generated by NRMAOR method converges to a unique solution $z^*\in \mathbb{R}^n$ for  any  initial vector $x^{(0)}\in \mathbb{R}^n$.
	\end{thm}
	\begin{proof}
		Let $\bar{A}=\langle M+\phi\rangle -\lvert N+\phi\rvert $, $\Omega_{1}\bar{A}=\bar{A}_{\Omega_{1}}$ and $A=(M+\phi)-(N+\phi)$ be an $H$-splitting. Therefore $\bar{A}$ is a nonsingular  $M$-matrix.\\
		$$0 \textless \bar{A} Ve \leq \langle M+\phi\rangle Ve \leq (\langle M+\phi\rangle +\Omega_{2}) Ve  $$
		and $$0 \textless \bar{A}Ve \textless \langle A\rangle Ve. $$
		Further
		\begin{equation}\label{eq8}
			\begin{split}
				&(\Omega_{2}+\langle M_{\Omega_{1}}
				\rangle+\lvert\phi_{\Omega_{1}}
				\rvert-\lvert N_{\Omega_{1}}+\phi_{\Omega_{1}}\rvert-\lvert \Omega_{2}-A_{\Omega_{1}}\rvert-2 \lvert A 
				\rvert \Psi \lvert A^{-1} 
				\rvert \Omega_{2})Ve\\&
				\geq 
				(\langle M_{\Omega_{1}}+\phi_{\Omega_{1}}\rangle-\lvert N_{\Omega_{1}}+\phi_{\Omega_{1}}\rvert+\Omega_{2}-\lvert \Omega_{2}-A_{\Omega_{1}}\rvert-2 \lvert A 
				\rvert \Psi \lvert A^{-1} 
				\rvert \Omega_{2})Ve.
			\end{split}
		\end{equation}
		Here, we consider two cases\\
		Case (1): when $\Omega_{2}\geq D_{\Omega_{1}}$,
		\begin{equation*}
			\begin{split}
				(\bar{A}_{\Omega_{1}}+\Omega_{2}-\lvert \Omega_{2}-A_{\Omega_{1}}\rvert-2 \lvert A 
				\rvert \Psi \lvert A^{-1} 
				\rvert \Omega_{2})Ve
				&= (\bar{A}_{\Omega_{1}}+\Omega_{2}-( \Omega_{2}-\langle A_{\Omega_{1}}\rangle)\\&-2 \lvert A 
				\rvert \Psi \lvert A^{-1} 
				\rvert \Omega_{2})Ve\\&
				= (\bar{A}_{\Omega_{1}}+\langle A_{\Omega_{1}}\rangle-2 \lvert A 
				\rvert \Psi \lvert A^{-1} 
				\rvert \Omega_{2})Ve\\&
				\geq 2(\bar{A}_{\Omega_{1}}-\lvert A 
				\rvert \Psi \lvert A^{-1} 
				\rvert \Omega_{2})Ve\\&
				\geq 0.
			\end{split}
		\end{equation*}
		Case (2): when $0 \textless \Omega_{2}\leq D_{\Omega_{1}}$,
		\begin{equation*}
			\begin{split}
				(\bar{A}_{\Omega_{1}}+\Omega_{2}-\lvert \Omega_{2}-A_{\Omega_{1}}\rvert\-2 \lvert A 
				\rvert \Psi \lvert A^{-1} 
				\rvert \Omega_{2})Ve
				&= (\bar{A}_{\Omega_{1}}+\Omega_{2}-(\lvert A_{\Omega_{1}}\rvert-\Omega_{2})\\&-2 \lvert A 
				\rvert \Psi \lvert A^{-1} 
				\rvert \Omega_{2})Ve\\&
				= (\bar{A}_{\Omega_{1}}+\Omega_{2}-(2D_{\Omega_{1}}-\langle A_{\Omega_{1}}\rangle-\Omega_{2})\\&-2 \lvert A 
				\rvert \Psi \lvert A^{-1} 
				\rvert \Omega_{2})Ve\\&
				\geq 2(\bar{A}_{\Omega_{1}}+\Omega_{2}-D_{\Omega_{1}}-\lvert A 
				\rvert \Psi \lvert A^{-1} 
				\rvert \Omega_{2})Ve\\&
				\geq 0.
			\end{split}
		\end{equation*}
		Now from Case (1), Case (2) and Equation  $(\ref{eq8})$, we obtain that 
		$$(\Omega_{2}+\langle M_{\Omega_{1}}
		\rangle+\lvert\phi_{\Omega_{1}}
		\rvert-\lvert N_{\Omega_{1}}+\phi_{\Omega_{1}}\rvert-\lvert \Omega_{2}-A_{\Omega_{1}}\rvert-2 \lvert A 
		\rvert \Psi \lvert A^{-1} 
		\rvert \Omega_{2})Ve ~\textgreater ~0.$$
		This implies that 
		$$(\Omega_{2}+\langle M_{\Omega_{1}}
		\rangle+\lvert\phi_{\Omega_{1}}
		\rvert)Ve~\textgreater~ (\lvert N_{\Omega_{1}}+\phi_{\Omega_{1}}\rvert+\lvert A_{\Omega_{1}}+\Omega_{2}\rvert+2 \lvert A 
		\rvert \Psi \lvert A^{-1} 
		\rvert \Omega_{2})Ve.$$
		Now from Lemma $\ref{2.4}$, we have 
		\begin{equation*}
			\begin{split}
				\rho(\bar{L})&=\rho(V^{-1}\bar{L}V)\\&
				\leq \|V^{-1}\bar{L}V\|_{\infty}\\&
				=\| V^{-1}(\Omega_{2}+\langle M_{\Omega_{1}}
				\rangle+\lvert\phi_{\Omega_{1}}
				\rvert)^{-1}(\lvert N_{\Omega_{1}}+\phi_{\Omega_{1}}\rvert+\lvert \Omega_{2}-A_{\Omega_{1}}\rvert
				\\&+2 \lvert A 
				\rvert \Psi \lvert A^{-1} 
				\rvert \Omega_{2})V\|_{\infty}\\&
				=\|((\Omega_{2}+\langle M_{\Omega_{1}}
				\rangle+\lvert\phi_{\Omega_{1}}
				\rvert)V)^{-1}(\lvert N_{\Omega_{1}}+\phi_{\Omega_{1}}\rvert+\lvert \Omega_{2}-A_{\Omega_{1}}\rvert\\&+2 \lvert A 
				\rvert \Psi \lvert A^{-1} 
				\rvert \Omega_{2})V\|_{\infty}\\&
				\leq \max_{1\leq i \leq n}\frac{((\lvert N_{\Omega_{1}}+\phi_{\Omega_{1}}\rvert+\lvert \Omega_{2}-A_{\Omega_{1}}\rvert+2 \lvert A 
					\rvert \Psi \lvert A^{-1} 
					\rvert \Omega_{2})Ve)_{i}}{((\Omega_{2}+\langle M_{\Omega_{1}}
					\rangle+\lvert\phi_{\Omega_{1}}
					\rvert)Ve)_{i}}\\&
				\textless 1.
			\end{split}
		\end{equation*}
		Hence, 	then the iterative sequence $\{z^{(k)}\}_ {k=0}^{+\infty} \subset \mathbb{R}^n$ generated by NRMAOR method converges to a unique solution $z^*\in \mathbb{R}^n$ for  any  initial vector $x^{(0)}\in \mathbb{R}^n$.
	\end{proof} 
	\section{Conclusion} 
	We discuss a class of new relaxation modulus-based matrix splitting methods based on a new matrix splitting approach for solving the implicit complementarity problem ICP$(q, A, \zeta)$ with parameter matrices  $\Omega_1$ and $\Omega_2$.  These methods are able to process the large and sparse structure of the matrix $A$.  For the proposed methods, we establish some convergence conditions when the system matrix is a $P$-matrix as well as some sufficient convergence conditions when the system matrix is an $H_{+}$-matrix.
	\section*{Declaration}
	\subsection*{Ethics approval and consent to participate} This article does not contain any studies with human participants or animals performed by any authors.
	\subsection*{Availability of data and materials} Not applicable.
	\subsection*{Conflict of interest} The authors declare that there are no conflicts of interest.
	\subsection*{Acknowledgments}
	The author, Bharat Kumar is thankful to the University Grants Commission (UGC), Government of India, under the SRF fellowship, Ref. No.: 1068/(CSIR-UGC NET DEC. 2017).
		\bibliographystyle{plain} 
	\bibliography{nam10}

\begin{thebibliography}{10}

\bibitem{8}
A~Bensoussan, M~Goursat, and JL~Lions.
\newblock Controle impulsionnel et inequations quasi-variationnelles
  stationnaires.
\newblock 1973.

\bibitem{7}
A~Bensoussan and JL~Lions.
\newblock Nouvelle formulation de probl{\`e}mes de contr{\^o}le impulsionnel et
  applications.
\newblock 1973.

\bibitem{53}
A~Berman and RJ~Plemmons.
\newblock Nonnegative matrices in the mathematical sciences.
\newblock {\em Bull. Amer. Math. Soc}, 6:233--235, 1982.

\bibitem{6}
D~Chan and JS~Pang.
\newblock The generalized quasi-variational inequality problem.
\newblock {\em Mathematics of Operations Research}, 7(2):211--222, 1982.

\bibitem{3}
R~Cottle, F~Giannessi, and JL~Lions.
\newblock {\em Variational inequalities and complementarity problems: theory
  and applications}.
\newblock John Wiley $\&$ Sons, 1980.

\bibitem{9}
RW~Cottle, JS~Pang, and RE~Stone.
\newblock {\em The Linear Complementarity Problem}, volume~60.
\newblock SIAM, 1992.

\bibitem{35}
AK~Das.
\newblock Properties of some matrix classes based on principal pivot transform.
\newblock {\em Annals of Operations Research}, 243:375--382, 2016.

\bibitem{45}
AK~Das, Deepmala, and R~Jana.
\newblock Some aspects on solving transportation problem.
\newblock {\em Yugoslav Journal of Operations Research}, 30(1):45, 2020.

\bibitem{36}
AK~Das and R~Jana.
\newblock Finiteness of criss-cross method in complementarity problem.
\newblock In {\em Mathematics and Computing: Third International Conference,
  ICMC 2017, Haldia, India, January 17-21, 2017, Proceedings 3}, pages
  170--180. Springer, 2017.

\bibitem{37}
AK~Das, R~Jana, and Deepmala.
\newblock On generalized positive subdefinite matrices and interior point
  algorithm.
\newblock In {\em Operations Research and Optimization: FOTA 2016, Kolkata,
  India, November 24-26 1}, pages 3--16. Springer, 2018.

\bibitem{18}
M~Dehghan and M~Hajarian.
\newblock Two class of synchronous matrix multisplitting schemes for solving
  linear complementarity problems.
\newblock {\em Journal of computational and applied mathematics},
  235(15):4325--4336, 2011.

\bibitem{20}
M~Dehghan, Dehghani MM, and M~Hajarian.
\newblock A two-step iterative method based on diagonal and off-diagonal
  splitting for solving linear systems.
\newblock {\em Filomat}, 31(5):1441--1452, 2017.

\bibitem{19}
M~Dehghan and A~Shirilord.
\newblock Matrix multisplitting picard-iterative method for solving generalized
  absolute value matrix equation.
\newblock {\em Applied Numerical Mathematics}, 158:425--438, 2020.

\bibitem{dutta2022column}
A~Dutta, R~Jana, and AK~Das.
\newblock On column competent matrices and linear complementarity problem.
\newblock In {\em Proceedings of the Seventh International Conference on
  Mathematics and Computing: ICMC 2021}, pages 615--625. Springer, 2022.

\bibitem{5}
CM~Elliott.
\newblock Variational and quasivariational inequalities applications to
  free--boundary problems.(claudio baiocchi and antonio capelo).
\newblock {\em SIAM Review}, 29(2):314, 1987.

\bibitem{31}
XM~Fang.
\newblock General fixed-point method for solving the linear complementarity
  problem.
\newblock {\em AIMS Mathematics}, 6(11):11904--11920, 2021.

\bibitem{32}
A~Frommer and DB~Szyld.
\newblock H-splittings and two-stage iterative methods.
\newblock {\em Numerische Mathematik}, 63:345--356, 1992.

\bibitem{17}
Xl~Han, DJ~Yuan, and S~Jiang.
\newblock Two saor iterative formats for solving linear complementarity
  problems.
\newblock {\em A A}, 1(1):1, 2011.

\bibitem{23}
J~He, H~Zheng, and S~Vong.
\newblock Improved inexact alternating direction methods for a class of
  nonlinear complementarity problems.
\newblock {\em EAST ASIAN JOURNAL ON APPLIED MATHEMATICS}, 12(1):125--144,
  2022.

\bibitem{22}
JW~He, CC~Lei, CY~Shi, and SW~Vong.
\newblock An inexact alternating direction method of multipliers for a kind of
  nonlinear complementarity problems.
\newblock {\em Numerical Algebra, Control and Optimization}, 11(3):353--362,
  2021.

\bibitem{50}
JT~Hong and CL~Li.
\newblock Modulus-based matrix splitting iteration methods for a class of
  implicit complementarity problems.
\newblock {\em Numerical Linear Algebra with Applications}, 23(4):629--641,
  2016.

\bibitem{54}
JG~Hu.
\newblock Estimates of $\| b^{-1}a\|_\infty$ and their applications.
\newblock {\em Math. Numer. Sin}, 4(3):272--282, 1982.

\bibitem{2}
G~Isac.
\newblock Complementarity problems.
\newblock 1992.

\bibitem{46}
R~Jana, AK~Das, and S~Sinha.
\newblock On processability of lemke’s algorithm.
\newblock {\em Applications and Applied Mathematics: An International Journal
  (AAM)}, 13(2):31, 2018.

\bibitem{13}
C~Kanzow.
\newblock Inexact semismooth newton methods for large-scale complementarity
  problems.
\newblock {\em Optimization Methods and Software}, 19(3-4):309--325, 2004.

\bibitem{47}
B~Kumar, Deepmala, Dutta A, and AK~Das.
\newblock Error bound for the linear complementarity problem using plus
  function.
\newblock {\em arXiv preprint arXiv:2209.00377}, 2022.

\bibitem{34}
B~Kumar, Deepmala, and AK~Das.
\newblock On general fixed point method based on matrix splitting for solving
  linear complementarity problem.
\newblock {\em Journal of Numerical Analysis and Approximation Theory},
  51(2):189--200, 2022.

\bibitem{44}
B~Kumar, Deepmala, A~Dutta, and AK~Das.
\newblock More on matrix splitting modulus-based iterative methods for solving
  linear complementarity problem.
\newblock {\em OPSEARCH}, pages 1--18, 2023.

\bibitem{49}
CL~Li and JT~Hong.
\newblock Modulus-based synchronous multisplitting iteration methods for an
  implicit complementarity problem.
\newblock {\em EAST Asian Journal on Applied Mathematics}, 7(2):363--375, 2017.

\bibitem{48}
N~Li, J~Ding, and JF~Yin.
\newblock Modified relaxation two-sweep modulus-based matrix splitting
  iteration method for solving a class of implicit complementarity problems.
\newblock {\em Journal of Computational and Applied Mathematics}, 413:114370,
  2022.

\bibitem{38}
SR~Mohan, SK~Neogy, and AK~Das.
\newblock More on positive subdefinite matrices and the linear complementarity
  problem.
\newblock {\em Linear Algebra and its Applications}, 338(1-3):275--285, 2001.

\bibitem{39}
SR~Mohan, SK~Neogy, and AK~Das.
\newblock On the classes of fully copositive and fully semimonotone matrices.
\newblock {\em Linear Algebra and its Applications}, 323(1-3):87--97, 2001.

\bibitem{4}
U~Mosco.
\newblock Implicit variational problems and quasi variational inequalities.
\newblock In {\em Nonlinear Operators and the Calculus of Variations: Summer
  School Held in Bruxelles 8--19 September 1975}, pages 83--156. Springer,
  2006.

\bibitem{10}
KG~Murty and FT~Yu.
\newblock {\em Linear complementarity, linear and nonlinear programming},
  volume~3.
\newblock Citeseer, 1988.

\bibitem{41}
SK~Neogy, RB~Bapat, AK~Das, and B~Pradhan.
\newblock Optimization models with economic and game theoretic applications.
\newblock {\em Annals of Operations Research}, 243:1--3, 2016.

\bibitem{43}
SK~Neogy and AK~Das.
\newblock Linear complementarity and two classes of structured stochastic
  games.
\newblock {\em Operations Research with Economic and Industrial Applications:
  Emerging Trends, eds: SR Mohan and SK Neogy, Anamaya Publishers, New Delhi,
  India}, pages 156--180, 2005.

\bibitem{neogy2006some}
SK~Neogy and AK~Das.
\newblock Some properties of generalized positive subdefinite matrices.
\newblock {\em SIAM journal on matrix analysis and applications},
  27(4):988--995, 2006.

\bibitem{neogy2011singular}
SK~Neogy and AK~Das.
\newblock On singular n0-matrices and the class q.
\newblock {\em Linear algebra and its applications}, 434(3):813--819, 2011.

\bibitem{42}
SK~Neogy, AK~Das, and A~Gupta.
\newblock Generalized principal pivot transforms, complementarity theory and
  their applications in stochastic games.
\newblock {\em Optimization Letters}, 6:339--356, 2012.

\bibitem{11}
MA~Noor.
\newblock Fixed point approach for complementarity problems.
\newblock {\em Journal of Mathematical Analysis and Applications},
  133(2):437--448, 1988.

\bibitem{1}
JS~Pang.
\newblock The implicit complementarity problem.
\newblock In {\em Nonlinear Programming 4}, pages 487--518. Elsevier, 1981.

\bibitem{27}
J~Rohn.
\newblock Description of all solutions of a linear complementarity problem.
\newblock {\em The Electronic Journal of Linear Algebra}, 18:246--252, 2009.

\bibitem{26}
J~Rohn.
\newblock On unique solvability of the absolute value equation.
\newblock {\em Optimization Letters}, 3(4):603--606, 2009.

\bibitem{28}
J~Rohn, V~Hooshyarbakhsh, and R~Farhadsefat.
\newblock An iterative method for solving absolute value equations and
  sufficient conditions for unique solvability.
\newblock {\em Optimization Letters}, 8:35--44, 2014.

\bibitem{21}
A~Shirilord and M~Dehghan.
\newblock Double parameter splitting (dps) iteration method for solving complex
  symmetric linear systems.
\newblock {\em Applied Numerical Mathematics}, 171:176--192, 2022.

\bibitem{14}
Z~Sun and J~Zeng.
\newblock A monotone semismooth newton type method for a class of
  complementarity problems.
\newblock {\em Journal of computational and applied mathematics},
  235(5):1261--1274, 2011.

\bibitem{52}
Y~Wang, JF~Yin, QY~Dou, and R~Li.
\newblock Two-step modulus-based matrix splitting iteration methods for a class
  of implicit complementarity problems.
\newblock {\em Numerical Mathematics: Theory, Methods \& Applications}, 12(3),
  2019.

\bibitem{12}
N~Xiu and J~Zhang.
\newblock Some recent advances in projection-type methods for variational
  inequalities.
\newblock {\em Journal of Computational and Applied Mathematics},
  152(1-2):559--585, 2003.

\bibitem{51}
H~Zheng and S~Vong.
\newblock A modified modulus-based matrix splitting iteration method for
  solving implicit complementarity problems.
\newblock {\em Numerical Algorithms}, 82:573--592, 2019.

\end{thebibliography}
\end{document}